\documentclass{article}

\usepackage{amsmath,amsthm,amssymb,amsfonts}

\usepackage{verbatim}
\usepackage{graphicx}

\usepackage{stmaryrd} 
\usepackage{hyperref}
\usepackage{graphicx}

\usepackage[colorinlistoftodos]{todonotes}

\newtheorem{thm}{Theorem}

\newtheorem{prop}[thm]{Proposition}

\newtheorem{assert}[thm]{Assertion}

\newtheorem{remarks}[thm]{Remark}

\newtheorem{definition}[thm]{Definition}

\newtheorem{exl}[thm]{Example}

\numberwithin{thm}{section}

\newcommand{\adj}{\leftrightarrow}

\newcommand{\adjeq}{\leftrightarroweq}

\DeclareMathOperator{\Fix}{Fix}

\def\Z{{\mathbb Z}}
\def\N{{\mathbb N}}

\begin{document}
\title{Remarks on Fixed Point Assertions in Digital Topology, 10}
\author{Laurence Boxer
\thanks{Department of Computer and Information Sciences, Niagara University, NY 14109, USA
and  \newline
Department of Computer Science and Engineering, State University of New York at Buffalo \newline
email: boxer@niagara.edu
\newline
ORCID: 0000-0001-7905-9643
}
}

\date{ }
\maketitle

\begin{abstract}
The topic of fixed points in digital metric spaces continues to draw
publications with assertions that are incorrect, incorrectly proven, trivial,
or incoherently stated. We continue the work of our earlier papers that
discuss publications with bad assertions concerning fixed points
of self-functions on digital images.

MSC: 54H25

Key words and phrases: digital topology, digital image,
fixed point, digital metric space
\end{abstract}

\section{Introduction}
Published assertions about fixed points in digital topology include
some that are beautiful and many that are incorrect, incorrectly
proven, or trivial. This paper continues the work
of~\cite{BxSt19, Bx19, Bx19-3, Bx20, Bx22, BxBad6, BxBad7, BxBad8, BxBad9}
in discussing flaws in papers that have come to our
attention since acceptance for publication of~\cite{BxBad9}.

We quote~\cite{BxBad8}:
\begin{quote}
... the notion of
a ``digital metric space" has led many authors to attempt,
in most cases either erroneously or trivially, to modify fixed
point results for Euclidean spaces to digital images. 
This notion contains roots of all the
flawed papers studied in the current paper.
See~\cite{Bx20} for discussion of why ``digital metric space"
does not seem a worthy topic of further research.
\end{quote}

\section{Preliminaries}
Much of the material in this section is quoted or
paraphrased from~\cite{Bx20}.

We use $\N$ to represent the natural numbers,
$\N^* = \N \cup \{0\}$, and
$\Z$ to represent the integers. 
The literature uses both $|X|$
and $\#X$ for the cardinality of~$X$. 

A {\em digital image} is a pair $(X,\kappa)$, where $X \subset \Z^n$ 
for some positive integer $n$, and $\kappa$ is an adjacency relation on $X$. 
Thus, a digital image is a graph.
In order to model the ``real world," we usually take $X$ to be finite,
although occasionally we consider
infinite digital images, e.g., for digital analogs of
covering spaces. The points of $X$ may be 
thought of as the ``black points" or foreground of a 
binary, monochrome ``digital picture," and the 
points of $\Z^n \setminus X$ as the ``white points"
or background of the digital picture.

\subsection{Adjacencies, 
continuity, fixed point}

In a digital image $(X,\kappa)$, if
$x,y \in X$, we use the notation
$x \adj_{\kappa}y$ to
mean $x$ and $y$ are $\kappa$-adjacent; we may write
$x \adj y$ when $\kappa$ can be understood. 
We write $x \adjeq_{\kappa}y$, or $x \adjeq y$
when $\kappa$ can be understood, to
mean 
$x \adj_{\kappa}y$ or $x=y$.

The most commonly used adjacencies in the study of digital images 
are the $c_u$ adjacencies. These are defined as follows.
\begin{definition}
\label{cu-adj-Def}
Let $X \subset \Z^n$. Let $u \in \Z$, $1 \le u \le n$. Let 
$x=(x_1, \ldots, x_n),~y=(y_1,\ldots,y_n) \in X$. Then $x \adj_{c_u} y$ if 
\begin{itemize}
    \item $x \neq y$,
    \item for at most $u$ distinct indices~$i$,
    $|x_i - y_i| = 1$, and
    \item for all indices $j$ such that $|x_j - y_j| \neq 1$ we have $x_j=y_j$.
\end{itemize}
\end{definition}

\begin{definition}
\label{path}
{\rm (See \cite{Khalimsky})} 
    Let $(X,\kappa)$ be a digital image. Let
    $x,y \in X$. Suppose there is a set
    $P = \{x_i\}_{i=0}^n \subset X$ such that
$x=x_0$, $x_i \adj_{\kappa} x_{i+1}$ for
$0 \le i < n$, and $x_n=y$. Then $P$ is a
{\em $\kappa$-path} (or just a {\em path}
when $\kappa$ is understood) in $X$ from $x$ to $y$,
and $n$ is the {\em length} of this path.
\end{definition}

\begin{definition}
{\rm \cite{Rosenfeld}}
A digital image $(X,\kappa)$ is
{\em $\kappa$-connected}, or just {\em connected} when
$\kappa$ is understood, if given $x,y \in X$ there
is a $\kappa$-path in $X$ from $x$ to $y$. The {\rm $\kappa$-component of~$x$ in~$X$} is the
maximal $\kappa$-connected subset
of~$X$ containing~$x$.
\end{definition}

\begin{definition}
{\rm \cite{Rosenfeld, Bx99}}
Let $(X,\kappa)$ and $(Y,\lambda)$ be digital
images. A function $f: X \to Y$ is 
{\em $(\kappa,\lambda)$-continuous}, or
{\em $\kappa$-continuous} if $(X,\kappa)=(Y,\lambda)$, or
{\em digitally continuous} when $\kappa$ and
$\lambda$ are understood, if for every
$\kappa$-connected subset $X'$ of $X$,
$f(X')$ is a $\lambda$-connected subset of $Y$.
\end{definition}

\begin{thm}
{\rm \cite{Bx99}}
A function $f: X \to Y$ between digital images
$(X,\kappa)$ and $(Y,\lambda)$ is
$(\kappa,\lambda)$-continuous if and only if for
every $x,y \in X$, if $x \adj_{\kappa} y$ then
$f(x) \adjeq_{\lambda} f(y)$.
\end{thm}

We use 
$C(X,\kappa)$ for the set of functions 
$f: X \to X$ that are $\kappa$-continuous.

A {\em fixed point} of a function $f: X \to X$ 
is a point $x \in X$ such that $f(x) = x$. We denote by
$\Fix(f)$ the set of fixed points of $f: X \to X$.

As a convenience, if $x$ is a point in the
domain of a function $f$, we will often
abbreviate ``$f(x)$" as ``$fx$".
Also, if $f: X \to X$, we use
``$f^n$" for the $n$-fold composition
\[ f^n = \overbrace{f \circ \ldots \circ f}^{n}
\]

\subsection{Digital metric spaces}
\label{DigMetSp}
A {\em digital metric space}~\cite{EgeKaraca-Ban} is a triple
$(X,d,\kappa)$, where $(X,\kappa)$ is a digital image and $d$ is a metric on $X$. The
metric is usually taken to be the Euclidean
metric or some other $\ell_p$ metric; 
alternately, $d$ might be taken to be the
shortest path metric. These are defined
as follows.
\begin{itemize}
    \item Given 
          $x = (x_1, \ldots, x_n) \in \Z^n$,
          $y = (y_1, \ldots, y_n) \in \Z^n$,
          $p > 0$, $d$ is the $\ell_p$ metric
          if \[ d(x,y) =
          \left ( \sum_{i=1}^n
          \mid x_i - y_i \mid ^ p
          \right ) ^ {1/p}. \]
          Note the special cases: if $p=1$ we
          have the {\em Manhattan metric}; if
          $p=2$ we have the 
          {\em Euclidean metric}.
    \item \cite{ChartTian} If $(X,\kappa)$ is a 
          connected digital image, 
          $d$ is the {\em shortest path metric}
          if for $x,y \in X$, $d(x,y)$ is the 
          length of a shortest
          $\kappa$-path in $X$ from $x$ to $y$.
\end{itemize}


We say a metric space $(X,d)$ is {\em uniformly discrete}
if there exists $\varepsilon > 0$ such that
$x,y \in X$ and $d(x,y) < \varepsilon$ implies $x=y$.

\begin{remarks}
\label{unifDiscrete}
If $X$ is finite or  
\begin{itemize}
\item {\rm \cite{Bx19-3}}
$d$ is an $\ell_p$ metric, or
\item $(X,\kappa)$ is connected and $d$ is 
the shortest path metric,
\end{itemize}
then $(X,d)$ is uniformly discrete.

For an example of a digital metric space
that is not uniformly discrete, see
Example~2.10 of~{\rm \cite{Bx20}}.
\end{remarks}

We say a sequence $\{x_n\}_{n=0}^{\infty}$ is 
{\em eventually constant} if for some $m>0$, 
$n>m$ implies $x_n=x_m$.
The notions of convergent sequence and complete digital metric space are often trivial, 
e.g., if the digital image is uniformly 
discrete, as noted in the following, a minor 
generalization of results 
of~\cite{HanBan,BxSt19}.

\begin{prop}
\label{eventuallyConst}
{\rm \cite{Bx20}}
If $(X,d)$ is a uniformly discrete metric space. 
then any Cauchy sequence in $X$
is eventually constant, and $(X,d)$ is a complete metric space.
\end{prop}

\subsection{A popular error}
\label{popularErrSec}
An error made by several authors, including a subset of those whose papers are
studied in the current work, is to confuse metric continuity (the
$\varepsilon - \delta$ version) and digital continuity.

Let $(X,d_X,\kappa)$ and $(Y,d_Y,\lambda)$ be digital metric spaces. 
Suppose~$d_X$ and~$d_Y$ are uniformly discrete. One sees easily
that every $f: X \to Y$ has metric continuity; however, depending on the
graph structures of $(X,\kappa)$ and $(Y,\lambda)$, there may be many
$f: X \to Y$ that fail to have digital continuity. The latter is true even
in many instances in which~$f$ is subject to additional restrictions, such as
being a digital contraction. Example~4.1 of~\cite{BxSt19} gives a counterexample to
the claim that a digital contraction mapping is digitally continuous.

\section{Banach contraction principle of \cite{EgeKaraca-Ban, Priyanka}}
The papers~\cite{EgeKaraca-Ban,Priyanka} state a digital version of the classic
Banach contraction principle. However, in each of these papers the ``proof" uses
the incorrect claim discussed in 
section~\ref{popularErrSec} about the digital continuity of a
digital contraction mapping .

We prove a slightly generalized version of the digital Banach contraction
principle in this section.

\begin{thm}
    \label{EKBanachThm}
    Let $(X,d,\kappa)$ be a uniformly discrete digital metric space.
    Let $f: X \to X$ be a digital contraction map. Then $f$ has a
    unique fixed point.
\end{thm}

\begin{proof}
    Since $f$ is a digital contraction map, there exists $k$, $0 \le k < 1$,
    such that $d(fx,fy) \le kd(x,y)$ for all $x,y \in X$. Let $x_0 \in X$
    and consider the sequence defined by $x_{n+1}= fx_n$. We have
    \[ d(x_{n+1},x_n) = d(fx_n,fx_{n-1}) \le k d(x_n, x_{n-1})
    \]
    so simple induction gives
    \[ d(x_{n+m},x_n) \le k^{m-1} d(x_{n+1}, x_n) \le k^{m+n}d(x_1,x_0)
       \to_{n \to \infty} 0.
       \]
    Thus, $\{x_n\}_{n=0}^{\infty}$ is a Cauchy sequence. 
    By Proposition~\ref{eventuallyConst}, there exists $z \in X$ such that
    for almost all~$n$, $z = x_n$. Therefore, for almost all~$n$,
\[ fz=fx_n = x_{n+1} = z,
    \]
    so $z$ is a fixed point of $f$.

    Uniqueness is shown as follows. Let $z, z' \in \Fix(f)$. Then
    \[ d(z,z') = d(fz,fz') \le kd(z,z')
    \]
    so $d(z,z') = 0$, i.e., $z=z'$.    
\end{proof}

\section{\cite{GopalEtal}'s compressions}
\subsection{Overall}
This is a paper whose English and other issues of presentation are
so bad that they may divert attention
from the inadequacy of its attempts at mathematics. Several times, where it seems
desirable to quote from~\cite{GopalEtal}, we use a picture of what
appears in the paper in order to allay fear of misquotation.

\begin{itemize} 
    \item The word ``proof" is used to introduce various statements better
          called ``definition," ``lemma," ``proposition," ``theorem," or
          ``remark."
    \item Where ``Proof" should appear, we see ``Definition".
    \item Labels that should be distinct are not. There are ``Proof 1" and
    ``Proof 2" in the ``Preliminaries" section and again
    in the ``Results" section.
    \item All apparent attempts to ``prove" something are concerned with 
          assertions that are not well defined as stated.
\end{itemize}
  
\subsection{Improper citations in ``Preliminaries" section}
``Proof 1" cites as its source for a definition of~$\Z^n$ a 2018 paper,
rather than a classic source.

``Proof 2" cites \cite{Bx99}, rather than the correct~\cite{Bx07}, as
the source of its definition for $k_m$ adjacency (what we call $c_m$ adjacency).
The definition is also misquoted, requiring $m < n$ rather than the correct
$m \le n$, where the image $X \subset \Z^n$ and $m$ is the maximum number
of coordinates in which adjacent points differ.

\subsection{``Preliminaries Proof 4"}
\begin{figure}
    \includegraphics{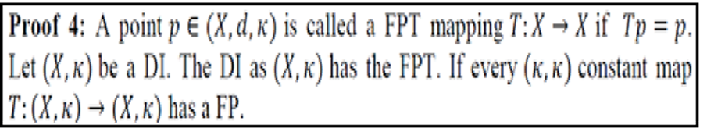}
    \caption{``Proof 4" of the ``Preliminaries" section of~\cite{GopalEtal}}
    \label{fig:GopalPrelimPrf4}
\end{figure}

``Proof 4" of the ``Preliminaries" section is shown in
Figure~\ref{fig:GopalPrelimPrf4}.
Abbreviations appearing in this ``Proof" are:

\[
\begin{array}{ll}
   \underline{Abbreviation} & \underline{Represents} \\
   FPT &  Fixed~Point~Theorem \\
   DI  & Digital~Image \\
   FP  & Fixed~Point
\end{array}
\]

What fixed point theorem is represented by ``FPT" is not explained.
Note the following.
\begin{itemize}
    \item Since when do we call a point a ``mapping"?
    \item Every constant self-map of~$X$ has a fixed point - its constant value.
    \item If the authors conclude that a digital image whose constant self maps
          have fixed points must have a fixed point property - i.e., that every
          self map has a fixed point - they are wrong, as demonstrated by
          the function $f: [0,1]_{\Z} \to [0,1]_{\Z}$ given by
          $f(x) = 1 - x$. If ``constant" is meant to be ``continuous" so that
          the intended statement is meant to define 
          {\em Fixed Point Property (FPP)},
          note that a digital image~$X$ has the 
          FPP if and only if $\#X = 1$~\cite{BEKLL}.
\end{itemize}

\subsection{\cite{GopalEtal}'s ``Results Proof 1"}
\cite{GopalEtal} calls upon the {\em Hausdorff metric}, a key to
Hundreds of papers in geometric topology. A classic reference on the 
{\em hyperspace} $(2^X, H)$ of compact subsets of~$X$ 
and the Hausdorff metric~$H$ based on~$d$, is~\cite{Nadler}.

``Proof 1" of the ``Results" section of~\cite{GopalEtal} seems to be an
attempt to state a theorem - see Figure~\ref{fig:GopalEtalPrf1}. Note the
following.
\begin{itemize}
    \item The left side of the inequality uses the metric~$H$ and
          the right side uses~$d$. It seems likely that the same
          metric should be used on both sides of the inequality - but
          which? While it seems probable that $d$ is the preferred metric,
          the first line of the ``Definition" [which seems, despite this
          designation, an attempt to prove the assertion] uses
          ``$\xi_1 \in T\xi_0$", suggesting that~$T$ is regarded as
          a multivalued function, hence $T\xi$ belongs to $2^X$ rather than
          to~$X$, and therefore that~$H$ really is the intended metric.

          Similar flaws appear in the other ``Proofs" 
          of~\cite{GopalEtal}'s "Results" section.
    \item It seems likely that the coefficient~$b$ appearing in the
          inequality is intended to coincide with the constant~$\theta_2$
          appearing on the following line.
\end{itemize}

Also, it seems likely that the symbols $\xi_n$ and $x_n$, that appear
in what seems to be an attempt at proving this assertion, are meant to
be identified - see Figure~\ref{fig:GopalEtalPrf1Err1}.
 
\begin{figure}
    \includegraphics{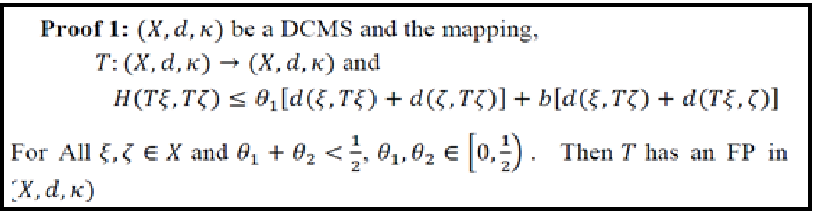}
    \caption{``Proof 1" of the ``Results" section of~\cite{GopalEtal}}
    \label{fig:GopalEtalPrf1}
\end{figure}

It seems the likely that the intended assertion of ``Proof 1" of the ``Results" section of~\cite{GopalEtal} is the following (we add that the fixed point is
unique).

\begin{thm}
    \label{uncorrectedGopalPrf1}
    Let $(X,d,\kappa)$ be a digital metric space, where~$d$ is
    uniformly discrete. Let $T: X \to X$ such that
    for all $x,y \in X$ and some nonnegative constants $a,b$ such that
    $0 \le a+b < 1/2$,
    \[ d(Tx,Ty) \le a[d(x,Tx) + d(y,Ty)] +b[d(x,Ty)+ d(Tx,y)].
    \]
    Then $T$ has a unique fixed point.
\end{thm}

\begin{figure}
    \includegraphics{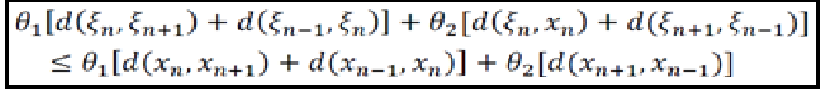}
    \caption{In the argument for ``Proof 1" of the ``Results" section
    of~\cite{GopalEtal}, it seems likely that the symbols $\xi_n$ and $x_n$ are
            identified for all~$n$.}
    \label{fig:GopalEtalPrf1Err1}
\end{figure}

\begin{proof}
    The argument given in~\cite{GopalEtal} can be corrected and
    abbreviated as follows.

    Let $x_0 \in X$, Inductively, let $x_{n+1}=Tx_n$. We have, for $n>0$,
    \[ d(x_n,x_{n+1}) = d(Tx_{n-1}, Tx_n) \le \]
    \[   a[d(x_{n-1},Tx_{n-1}) + d(x_n,Tx_n)] + b[d(x_{n-1},Tx_n) + 
                  d(Tx_{n-1},x_n)] =
    \]
    \[ a[d(x_{n-1},x_n)+ d(x_n,x_{n+1})] + b[d(x_{n-1},x_{n+1}) + d(x_n,x_n)] \le
    \]
    \[ a[d(x_{n-1},x_n)+ d(x_n,x_{n+1})] + b[d(x_{n-1},x_n) + d(x_n,x_{n+1})+0 ] =
    \]
    \[ (a+b) [d(x_{n-1},x_n)+ d(x_n,x_{n+1})].
    \]
    So
    \[ [1 - (a+b)] d(x_n,x_{n+1}) \le (a+b) d(x_{n-1},x_n),
    \]
    or, for $0 \le A = \frac{a+b}{1-(a+b)} < 1$, 
    $d(x_n,x_{n+1}) \le Ad(x_{n-1},x_n)$, and by a simple induction,
    \[ d(x_n,x_{n+1}) \le A^n d(x_0,x_1).
    \]
    Therefore, for $m > n$
    \[ d(x_n,x_m) \le \sum_{i=0}^{m-1} d(x_{n+i},x_{n+i+1}) \le 
       \sum_{i=0}^{m-1}A^{n+i}d(x_0,x_1) \le 
    \]
    \[   A^m d(x_0,x_1) \to_{n \to \infty} 0.
    \]
    Thus $\{x_n\}_{n=0}^{\infty}$ is a Cauchy sequence.
    By Proposition~\ref{eventuallyConst}, we have that, for almost all~$n$,
    $x_n =  x_{n+1} = Tx_n$. Thus, $T$ has a fixed point.

To show uniqueness, suppose $u,v \in \Fix(T)$. Then
\[ d(u,v) = d(Tu,Tv) \le 
\]
\[ a[d(u,Tu) + d(v,Tv)] + b[d(u,Tv) + d(Tu,v)] = 
\]
\[ a[0+0] +b[d(u,v) + d(u,v)] = 2b \cdot d(u,v).
\]
Since $2b < 1$, we have $d(u,v) = 0$, i.e., $u=v$.
\end{proof}

\subsection{\cite{GopalEtal}'s ``Results Proof 2"}
\begin{figure}
    \includegraphics{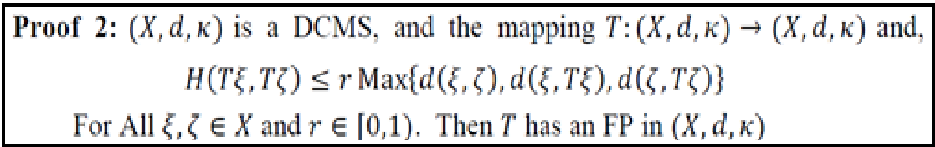}
    \caption{The assertion ``Proof 2" of the ``Results"
                 of~\cite{GopalEtal}}
    \label{fig:GopEtalRltsPrf2.eps}
\end{figure}

``Proof 2" of the ``Results" of~\cite{GopalEtal} appears 
in Figure~\ref{fig:GopEtalRltsPrf2.eps}. This assertion and
the argument offered as its ``proof" seem to have
the flaws mentioned above of confusing the metric~$d$ with the
Hausdorff metric, confusing single-valued and
multiple-valued functions, and identifying the symbols~$\xi_n$
and~$x_n$. One suspects the
assumption of this ``Proof" is the following.

\begin{assert}
    \label{uncorrectedGopPrf2}
    Let $(X,d,\kappa)$ be a digital metric space. Let $T: X \to X$
    satisfy, for $0 \le r < 1$ and all $\xi, \zeta \in X$,
    \[ d(T\xi,T\zeta) \le r \max\{d(\xi,\zeta), d(\xi, T\xi), d(\zeta, T\zeta)\}.
    \]
    Then $T$ has a fixed point.
\end{assert}

    \begin{figure}
    \includegraphics{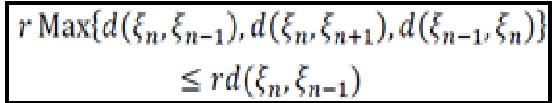}
    \caption{Unjustified claim in ``proof" of ``Proof 2" in ``Results"
    of~\cite{GopalEtal}}
    \label{fig:GopPrf2Err.eps}
\end{figure}

The argument given in~\cite{GopalEtal} as ``proof" of 
Assertion~\ref{uncorrectedGopPrf2} reaches lines marked ``(30)" and ``(31)",
where the inequality shown in Figure~\ref{fig:GopPrf2Err.eps} appears. For
this inequality to be valid, we need to be able to assume that
\begin{equation}
\label{GopPrf2Ineq}
    d(\xi_{n+1}, \xi_n) \le d(\xi_n,\xi_{n-1})
\end{equation}
as part of an attempt
to show that $\{\xi_n\}_{n=0}^{\infty}$ is a Cauchy sequence.
As no reason is given to justify the claim~(\ref{GopPrf2Ineq}),
we conclude that Assertion~\ref{uncorrectedGopPrf2} is unproven.

\subsection{\cite{GopalEtal}'s ``Results Proof 3.3"}
``Proof 3.3" and the argument for its ``proof" in~\cite{GopalEtal} 
share blemishes with previous assertions: confusing the metric~$d$ with the
Hausdorff metric, confusing single-valued and
multiple-valued functions, and, in the attempted proof,
identifying the symbols~$\xi_n$
and~$x_n$. See Figures~\ref{fig:GopPrf3-3.eps}
and~\ref{fig:Gop3.3Ineq.eps}. We also note that
``$P$ a standard shape of a circle with fixed integer '$K$'" is 
unexplained and not apparently called for in the ``Definition" [attempted
proof].

    \begin{figure}
    \includegraphics{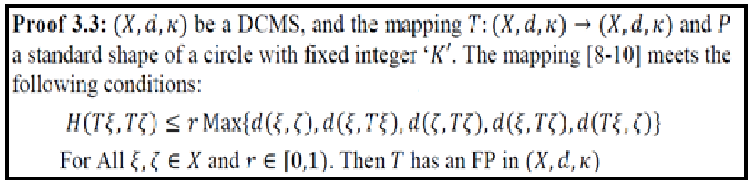}
    \caption{Assertion of ``Proof 3.3" in ``Results" of~\cite{GopalEtal}}
    \label{fig:GopPrf3-3.eps}
\end{figure}

One suspects the assumption of this ``Proof" is the following.

\begin{assert}
    \label{uncorrectedGop3-3}
        Let $(X,d,\kappa)$ be a digital metric space. Let $T: X \to X$
    satisfy, for $0 \le r < 1$ and all $\xi, \zeta \in X$,
    \[ d(T\xi,T\zeta) \le r \max\{d(\xi,\zeta), d(\xi, T\xi), d(\zeta, T\zeta),
       d(\xi,T\zeta), d(T\xi, \zeta) \}.
    \]
    Then $T$ has a fixed point.
\end{assert}

    \begin{figure}
    \includegraphics{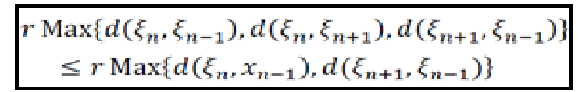}
    \caption{Unjustified inequality in ``Proof 3.3" in ``Results" of~\cite{GopalEtal}}
    \label{fig:Gop3.3Ineq.eps}
\end{figure}

The argument offered in proof of this assertion displays the
following errors.
\begin{itemize}
   \item We reach lines marked ``(44)" and~``(45)", where the 
         inequality shown in Figure~\ref{fig:Gop3.3Ineq.eps} appears.
         Even when we assume $\xi_{n-1}=x_{n-1}$, no justification is
         offered for dropping $d(\xi_n, \xi_{n+1})$ as a parameter of
         the $\max$ function.
    \item The argument develops (with an error
that may be minor) an inequality
\[ d(\xi_{n+1},\xi_n) \le r \max\{d(\xi_n,\xi_{n-1}),d(\xi_{n+1},\xi_{n-1}) \}
\]
and proceeds, assuming that either
\begin{itemize}
    \item $d(\xi_{n+1},\xi_n) \le r d(\xi_n,\xi_{n-1})$ for all~$n$, or
    \item $d(\xi_{n+1},\xi_n) \le r d(\xi_{n+1},\xi_{n-1}) $ for all~$n$.
\end{itemize} 
No justification is given for the assumption that the same case is valid
for all~$n$.
\end{itemize}
We conclude that Assertion~\ref{uncorrectedGop3-3} is unproven.

\section{\cite{MishTrip}'s mappings}
The title of~\cite{MishTrip} is misleading, as this paper makes no
original assertions about intimate mappings. We therefore do not present
a definition of an intimate mapping.
The paper has many improper citations and flaws in its mathematics.

\subsection{Improper citations and related errors}
Many notions appearing in~\cite{MishTrip}
are attributed to non-original sources.
These include the following.

\[
\begin{array}{lll}
\underline{Notion} & \underline{\cite{MishTrip}~attrib.}
& \underline{Better~attrib.} \\
Def.~1: c_u-adjacent~(k-adj.)
& \cite{MishraEtAl19} & \cite{BxHtpyProps} \\
\#~c_u~neighbors~in~\Z^n & \cite{MishraEtal21} & \cite{Han08} \\
Def.~2:~digital~interval & \cite{KongPi1},~\cite{MishraEtAl19} & \cite{Bx94} \\
Def.~3:~dig.~neighbor,~nbhd. & \cite{MishraEtal21} & \cite{Rosenfeld}
\\ Def.~5:~dig.~continuous~funct., \\~~~~~~~~~~~isomorph. & \cite{MishraEtal21} &
\cite{Rosenfeld},~\cite{Bx94}~(homeomorph.) \\
Def.~6:~dig.~path & \cite{MishraEtal21} & \cite{Rosenfeld} \\
Def.~7:~metric & \cite{EgeKaraca-Ban} & classic
\end{array}
\]

Also note~\cite{MishTrip}'s Definition~7 has its metric function
incorrectly taking its values in~$\Z^n$ rather than in~$\Z$.

Other improper citations:

\[
\begin{array}{lll}
\underline{Notion} & \underline{\cite{MishTrip}~attrib.}
& \underline{Better~attrib.} \\
Def.~8:~Cauchy~seq. & \cite{MishraEtal21} & classic \\
Notion~1:~Cauchy~seq.~eventually~const. & \cite{MishraEtAl19} & \cite{HanBan} \\
Notion~2:~dig.~complete~metric~space & \cite{MishraEtAl19} & classic,~\cite{EgeKaraca-Ban} \\
Def.~10:~dig.~contraction~map & \cite{MishraEtal21} & \cite{EgeKaraca-Ban}
\end{array}
\]

A remark at the end of Definition~10 of~\cite{MishTrip} makes the mistake
discussed in section~\ref{popularErrSec}.

In Notion~4 of~\cite{MishTrip}, the concept of weakly commuting maps
is undefined.

Theorem~2.1 of~\cite{MishTrip}, the Banach Fixed Point Theorem of
digital topology, is incorrectly attributed to~\cite{MishraEtAl19}.
The statement of the theorem should be attributed to~\cite{EgeKaraca-Ban}, where,
as noted above, the ``proof" is incorrect.

Following the statement of Notion~5 in~\cite{MishTrip} is a claim that
``Converse of Notion 5 is not followed, in general." Neither example nor citation
is given to support this claim.

As most of the inappropriate attributions in~\cite{MishTrip} cite other
papers for which Mishra and Tripathi are also coauthors, there is present a
suggestion that they are claiming credit for the achievements of others.

\subsection{\cite{MishTrip}'s ``Theorem" 3.1}
``Theorem"~3.1 of~\cite{MishTrip} is stated below, somewhat paraphrased.
In this section, we remark on the errors in its statement and ``proof" and show
that the assertion is false.

\begin{assert}
\label{MishTrip3.1}
    Suppose $\emptyset \neq Y \subset \Z^n$ where $(Y,\theta, \kappa)$ is
    a digital metric space. Let $G,H: Y \to Y$ such that 
    $H(Y) \subset G(Y)$ and
    \[\theta(Hx,Hy) \le \rho \theta(Gx,Gy) \mbox{ for some constant } \rho,
    ~0 \le \rho < 1 \mbox{, and all } x,y \in Y.
    \]
    Then $G$ and $H$ have a unique common fixed point.
\end{assert}

Flaws in the argument given for this assertion:
\begin{itemize}
    \item The word ``unique" was omitted from the statement of this ``theorem"
in~\cite{MishTrip} and is claimed in the ``proof" without any attempt at justification.
\item A statement ``$H(x_n) \subset G(x_{n+1})$" should be ``$H(x_n) = G(x_{n+1})$".
\item A statement
\[ \theta(Gx_{n+1}, Gx_n) \le \rho^n d(Gx_1, Gx_0)
\]
should be
\[ \theta(Gx_{n+1}, Gx_n) \le \rho^n \theta(Gx_1, Gx_0)
\]
\item The authors arrive at $Hx_n = Gx_{n+1} \to_{n \to \infty} z \in Y$. They claim
      this implies~$z$ is a common fixed point of~$G$ and~$H$. No reason is given for this
      claim, which this writer does not see how to justify. Note that the claim represents
      three unproven subclaims: that~$G$ has a fixed point, that~$H$ has a fixed point, 
      and that these fixed points coincide.
\end{itemize}

In fact, Assertion~\ref{MishTrip3.1} is incorrect, as shown by the following.

\begin{exl}
    \label{MishTripCounterexample}
    Let $G,H: \Z \to \Z$ be the functions
    \[ H(x) = 0, ~~~~ G(x) = x+1.\]
    Let $\theta(x,y) = |x - y|$.
    Then $H(\Z) \subset G(\Z)$ and
    \[ \theta(Hx, Hy) = 0 \le 0.5 \cdot |x - y| = 0.5 \cdot \theta(x,y) 
    \]
    for all $x,y \in \Z$. However, $G$ has no fixed point.
\end{exl}

\subsection{\cite{MishTrip}'s ``Theorem" 3.2}
``Theorem"~3.2 of~\cite{MishTrip} has a subset of the assumptions of
Assertion~\ref{MishTrip3.1}, making it easy for us to show that this
assertion is false. The assertion is stated below.

\begin{assert}
\label{MishTrip3.2}
    Suppose $\emptyset \neq Y \subset \Z^n$ where $(Y,\theta, \kappa)$ is
    a digital metric space. Let $G,H: Y \to Y$ such that 
    \[\theta(Hx,Hy) \le \rho \theta(Gx,Gy) \mbox{ for some constant } \rho,
    ~0 \le \rho < 1 \mbox{, and all } x,y \in Y.
    \]
    Then $G$ and $H$ have a unique common fixed point.
\end{assert}

\begin{itemize}
    \item An example is given in~\cite{MishTrip} as ``proof" of this assertion. 
          Mathematicians should know that an example does not prove a more
          general assertion.
    \item Our Example~\ref{MishTripCounterexample} provides a counterexample to
          Assertion~\ref{MishTrip3.2}.
\end{itemize}
 
\subsection{\cite{MishTrip}'s ``Corollary"~3.1}
We use the following.

\begin{definition}
{\rm \cite{DalalEtal}}
\label{compatibleDef}
    Let $S,T: X \to X$, where $(X,d,\kappa)$ is a digital metric space.
    If for every $\{x_n\}_{n=0}^{\infty} \subset X$ such that
    \[ \lim_{n \to \infty} Sx_n = \lim_{n \to \infty} Tx_n = t \in X
    \]
    we have
    \[ \lim_{n \to \infty} d(STx_n, TSx_n) = 0
    \]
    then $S$ and $T$ are {\em compatible}.
\end{definition}

We show below that \cite{MishTrip}'s ``Corollary"~3.1 is not correctly proven.
The assertion is the following.

\begin{assert}
\label{MishTripCor}
    Let $G,H: Y \to Y$, where $(Y,\theta, k)$ is a digital metric space and
    $G$ and $H$ are strictly increasing {\em [does this mean $Y \subset \Z$?]}. Suppose
    $H(Y) \subset G(Y)$ and, for all $x,y \in Y$ and some constant $\rho \in [0,1)$,
    \[ \theta(Hx,Hy) \le \rho \theta(Gx,Gy).
    \]
    Then $G$ and $H$ are compatible.
\end{assert}

The argument used to ``prove" this assertion in~\cite{MishTrip} depends on
Assertion~\ref{MishTrip3.1}, which we showed, via Example~\ref{MishTripCounterexample},
is false. Therefore, Assertion~\ref{MishTripCor} must be regarded as unproven. 

\section{\cite{ParvRaman}'s common fixed point ``theorem"}
In this section, we discuss an incorrect assertion of~\cite{ParvRaman}. We have 
the following, stated as the only new ``theorem" of~\cite{ParvRaman}.

\begin{assert}
    {\rm \cite{ParvRaman}}
    \label{ParvRamanAssert}
    Let $(X,d,\kappa)$ be a complete digital metric space. 
    Let~$T,S: X \to X$ be continuous (the argument in the ``proof" clarifies that this
    means in the $\varepsilon - \delta$ sense). Suppose for all $x,y \in X$,
    \begin{equation}
    \label{ParvRamanIneq}
    d(Tx,Sy) \le \frac{d(x,Tx)d(x,Sy) + d(y,Sy)d(y,Tx)}{d(x,Sy) + d(y,Tx)}
    \end{equation}
If for some $x_0 \in X$ the sequence
\[ x_{2n+1} = Tx_{2n},~~~~~ x_{2n+2} = Sx_{2n+1}
\]
has a subsequence $\{x_{n_k}\}$ converging to $a \in X$, then~$a$ is a unique
common fixed point of~$S$ and~$T$.
\end{assert}

The argument given as ``proof" for Assertion~\ref{ParvRamanAssert} has
the following flaws.
\begin{itemize}
    \item The authors appear to have confused themselves with deceptive notation.
          They used ``$x_{nk}$" rather than ``$x_{n_k}$" for members of the
          subsequence. This appears to have led to their conclusion that the sequence with
          even subscripts $x_{2nk}$ converges. We have no reason to believe
          that $\{x_{2 \cdot n \cdot k}\}$ is a subsequence of $\{x_{2 \cdot n_k}\}$. However, the use of better notation seems to render this flaw as not serious.
    \item More important: towards the end of the ``proof," the authors claim that
          each of the following is a contradiction.
    \begin{itemize}
        \item ~~~~~$d(a,Ta) = d(Ta,STa) \le d(a,Ta)$
        \item ~~~~~$d(a,Sa) = d(Ta,STa) \le d(a,Sa)$
        \item ~~~~~Assuming $a=a'$, then $d(a,a') \le \ldots \le 0$
    \end{itemize}
        Clearly, none of these is a contradiction.
\end{itemize}

Further, we have the following. 
\begin{prop}
    Statement~{\rm (\ref{ParvRamanIneq})}is not well-defined.
\end{prop}
\begin{proof}
    Suppose there exists a common fixed point~$x_0$. Then for $x=y=x_0$, the right
side of~(\ref{ParvRamanIneq}) becomes
\[ \frac{d(x_0,Tx_0) d(x_0,Sx_0) + d(x_0,Sx_0)d(x_0,Tx_0)}{d(x_0,Sx_0) + d(x_0,Tx_0)}
\]
Since $x_0$ is a common fixed point, this gives a right side of the form $\frac{0}{0}$,
which is undefined.
\end{proof}

\section{\cite{SalujaEtal}'s new contractive framework}
The paper~\cite{SalujaEtal} uses the notion of weakly commutative maps, defined
as follows.

\begin{definition}
    {\rm \cite{RaniEtAl}} Let $(X,d,\kappa)$ be a digital metric space.
    Let $S,T: X \to X$ such that
   \[     d(S(T(x)),T(S(x))) \le d(S(x),T(x)) \mbox{  for all  } x \in X.
   \]
    Then $S$ and $T$ are {\em weakly commutative} mappings.
\end{definition}

The only new assertion of~\cite{SalujaEtal} is the paper's
``Theorem"~3.1, which we show to be both trivial and false. This assertion,
slightly paraphrased, is as follows.

\begin{assert}
\label{SalujaEtal3.1}
    Let $(X,d,\kappa)$ be a digital metric space, where $d$ is the Euclidean metric.
    Let $J,K: X \to X$ such that
    \begin{enumerate}
        \item $K$ is continuous;
        \item $(J,K)$ is a weakly commutative pair of functions; and
        \item $J,K$ satisfy
       \[ \mbox{for all } u,q \in X \mbox{ and some } \xi 
            \mbox{ satisfying } 0 \le \xi < 1, 
            \]
        \begin{equation}
        \label{SalIneq}
            d(Ju,Jq) + d(Ku,Kq) \le \xi d(Ku,Kq).      
        \end{equation}
    \end{enumerate}
    Then $J$ and $K$ have a common fixed point.
\end{assert}

Property~(\ref{SalIneq}) is the ``new framework" in the title of~\cite{SalujaEtal}.

We need not use the assumption that $d$ is the Euclidean metric, nor the assumption
that~$K$ is continuous,
in order to show the flaws of Assertion~\ref{SalujaEtal3.1}. We have the
following.

\begin{prop}
    Let $(X,d)$ be a metric space. Let $J,K: X \to X$
    satisfying~(\ref{SalIneq}).
    Then~$J$ and $K$ are constant functions. Further, the
    conclusion of Assertion~\ref{SalujaEtal3.1} is true for all
    such $J,K$ if and only if $\#X = 1$.
\end{prop}

\begin{proof}
    The property~(\ref{SalIneq}) implies
    \[ 0 \le d(Ju,Jq) \le (\xi - 1) d(Ku,Kq) \le 0
    \]
    Hence, $d(Ju,Jq) = 0 = d(Ku, Kq)$, 
    so~$J$ and~$K$ are constant functions.

    Trivially, $\#X = 1$ implies $J,K: X \to X$ are functions with a
    unique common fixed point.
    Suppose $X$ has distinct points $x,y$. Consider the constant functions
    $Ju=x$ and $Ku=y$. One sees easily that this pair 
    of functions satisfies the hypotheses 
    of~Assertion~\ref{SalujaEtal3.1}, yet has no common fixed point.
\end{proof}

\section{\cite{ShaheenEtAl}'s iterative algorithms}
The paper~\cite{ShaheenEtAl} attempts to obtain results suggested by the paper's
title for digital images. We discuss why this is a dubious endeavor, and we 
consider the assertions put forth in~\cite{ShaheenEtAl}.

\subsection{On the wisdom of pursuing this line of research}
A {\em stable iteration procedure} is described as follows~\cite{HarderHicks}.
Let $(X,d)$ be a metric space, $T: X \to X$. Let $x_0 \in X$ and
let $x_{n+1} = f(T,x_n)$ be an {\em iteration procedure} for some function $f$; e.g.,
we might have $f(T,x) = Tx$. Suppose $x_n \to_{n \to \infty} p \in \Fix(T)$. 
Let $\{y_n\}_{n=0}^{\infty} \subset X$. Let $\varepsilon_n = d(y_{n+1}, f(T,y_n))$.
If $\varepsilon_n \to_{n \to \infty} 0$ implies $y_n \to_{n \to \infty} p$, then the
iteration procedure $x_{n+1} = f(T,x_n)$ is {\em $T$-stable} or {\em stable with
respect to} $T$.

Harder and Hicks~\cite{HarderHicks} note that $\{y_n\}_{n=0}^{\infty}$ might typically arise
via approximations due to rounding or discretization in calculating the function~$T$,
so we end up with $y_n$ as an approximation of $x_n$. I.e., start with $y_0 = x_0$ and
$y_{n+1} = f(T,y_n)$. If $f(T,x_n)$ is $T$-stable and $x_n \to_{n \to \infty} p \in \Fix(T)$,
then, claim Harder and Hicks, if $\varepsilon_n \to_{n \to \infty} 0$ then $y_n \to_{n \to \infty} p$.

The description above of a stable iteration procedure seems to belong
      naturally in the world of real analysis. By contrast, digital metric spaces
      live in~$\Z^n$ and typically use uniformly discrete metrics,
      where it seems likely that rounding and discretization errors
      will often be nil, yielding $y_n = x_n$.

\subsection{\cite{ShaheenEtAl}'s inappropriate citations and related errors}
\cite{ShaheenEtAl} has several notions that are not properly quoted or attributed. 
These include the following.

Definition~1 (what we call $c_u$ adjacency - see our Definition~\ref{cu-adj-Def}): 
``... if there is $r$ indices ..." should be ''if there are at most $r$ indices ..."

\[
\begin{array}{lll}
\underline{Notion} & \underline{\cite{ShaheenEtAl}~attrib.}
& \underline{Better~attrib.} \\
Def.~5~(digital~interval) & \cite{EgeKaraca-Ban} & \cite{Bx94} \\
Def.~6~(connected~dig.~img.) &  \cite{EgeKaraca-Ban} & \cite{Rosenf79} \\
Def.~7~(dig.~metric~sp.) & \cite{SriEtal} & \cite{EgeKaraca-Ban}
\end{array}
\]

``Proposition"~1~\cite{EgeKaraca-Ban} has the error discussed in
section~\ref{popularErrSec}, claiming incorrectly
that a digital contraction map must be digitally continuous.

Theorem~1~\cite{HanBan} claims that a digital metric space is a complete
metric space. While this is true under an additional assumption, a counterexample
appears at Example~2.9 of~\cite{BxSt19}.

``Definition"~1 is, at best, incomplete; it doesn't define anything.

Lemma~1~\cite{Berinde} is incorrectly quoted as follows. A sequence
$\{\varepsilon_n\}_{n=0}^{\infty}$ of real numbers is hypothesized. It should be stated that
$\varepsilon_n \to_{n \to \infty} 0$, but, instead, $\nu_n \to_{n \to \infty} 0$ is stated.
It is easily seen that the misstated version of the lemma in~\cite{ShaheenEtAl} is false.

Definition~12~\cite{BotmartEtAl} seeks to define a sequence 
\[ t_{n+1} = f_{F,\alpha_n}(t_n)
\]
although there is no definition in the paper that can be applied to the
symbol $f_{F,\alpha_n}$. Since Definition~12 is meant to be the definition of
{\em stability} that is the focus of the rest of the paper, the merits of the entire
paper are in question. Consulting the ``definition's" source, the paper~\cite{BotmartEtAl},
offers no help since, as noted in~\cite{BxBad8}, the same omission occurs there.

\subsection{\cite{ShaheenEtAl}'s Example 1}
The authors seek to use in this example a function 
$F: (\N^*, \tilde{\mu}, \rho) \to (\N^*, \tilde{\mu}, \rho)$ given by
\[ F(t) = \frac{t}{2} + 1
\]
where the triple $(\N^*, \tilde{\mu}, \rho)$ is a digital image based on
the nonnegative integers $\N^*$. But this is inappropriate, since~$F$
is not integer-valued.

\section{Further remarks}
We paraphrase~\cite{BxBad8}:
\begin{quote}
We have discussed several papers that seek to advance
fixed point assertions for digital metric spaces.
Many of these assertions are incorrect, incorrectly proven, 
or reduce to triviality; consequently, all of these
papers should have been rejected or required to undergo
extensive revisions. This reflects badly not only on
the authors, but also on the referees and editors who
approved their publication.
\end{quote}

\end{document}